\documentclass{amsart}
\usepackage{mathrsfs}
\usepackage{amsfonts}
\usepackage{amsmath,amssymb,amsfonts}
\usepackage{setspace}
\usepackage{dsfont}

 \newtheorem{thm}{Theorem}[section]
 \newtheorem{cor}[thm]{Corollary}
 \newtheorem{lem}[thm]{Lemma}
 \newtheorem{prop}[thm]{Proposition}
 \theoremstyle{definition}
 \newtheorem{defn}[thm]{Definition}
 \theoremstyle{remark}
 
 \newtheorem*{ex}{Example}
 \numberwithin{equation}{section}
 
\begin{document}

%
%
%
%
%
%
%
%
%

\title[Direct Integrals of Strongly Irreducible Operators]
 {Direct Integrals of Strongly Irreducible Operators}

\author[C. Jiang]{Chunlan Jiang}

\address{
Department of Mathematics\\
Hebei Normal University\\
Shijiazhuang 050016\\
P.R.China}

\email{cljiang@hebtu.edu.cn}


\author[{R. Shi}]{Rui Shi$^{\ast}$}
\address{
Department of Mathematics\\
Hebei Normal University\\
Shijiazhuang 050016\\
P.R.China}

\email{littlestoneshg@gmail.com}

\thanks{The authors were supported in part by NSFC Grant (No.10731020).\\
$^{\ast}${Corresponding author}}





\begin{abstract}

Strongly irreducible operators can be considered as building blocks
for bounded linear operators on complex separable Hilbert spaces.
Many bounded linear operators can be written as direct sums of at
most countably many strongly irreducible operators. In this paper,
we show that a bounded linear operator $A$ is similar to a direct
integral of strongly irreducible operators if its commutant $\{A\}'$
contains a bounded maximal abelian set of idempotents. We find that
bounded linear operators which are similar to direct integrals of
strongly irreducible operators form a dense subset of $\mathscr
{L}(\mathscr {H})$ in the operator norm.

\end{abstract}

\maketitle

\section{Introduction}

Throughout this paper, all Hilbert spaces discussed are
\textit{complex and separable} while all operators are bounded
linear operators on the Hilbert spaces. Let $\mathscr {H}$ be a
Hilbert space, and let $\mathscr {L}(\mathscr {H})$ be the set of
bounded linear operators on $\mathscr {H}$. Every projection here is
self-adjoint while an idempotent may be not. Unless otherwise
stated, the term \textit{algebra} will always refer to a unital
subalgebra of $\mathscr {L}(\mathscr {H})$ which is closed in the
strong operator topology.

An operator $A$ on a Hilbert space $\mathscr {H}$ is said to be
{\textit{irreducible}} if its commutant $\{A\}'\equiv \{B\in\mathscr
{L}(\mathscr {H}):AB=BA\}$ contains no projections other than $0$
and $I$, introduced by P. Halmos in \cite{Halmos}. (The separability
assumption is necessary because on a non-separable Hilbert space
every operator is reducible.) An operator $A$ on a Hilbert space
$\mathscr {H}$ is said to be {\textit{strongly irreducible}} if
$XAX^{-1}$ is irreducible for every invertible operator
$X\in\mathscr {L}(\mathscr {H})$ \cite{Gilfeather}.

Strongly irreducible operators are irreducible but the converse is
not true in general.  Jordan blocks, the standard unilateral shift
operator, unicellular operators and Cowen-Douglas operators of index
1 are classical examples of strongly irreducible operators.

On finite dimensional Hilbert spaces, every operator can be uniquely
written as an orthogonal direct sum of irreducible operators up to
unitary equivalence. Furthermore, the Jordan canonical form theorem
shows that every operator on a finite dimensional Hilbert space can
be uniquely written as a (Banach) direct sum of Jordan blocks up to
similarity. A natural question is whether these results can be
generalized to operators on infinite dimensional Hilbert spaces. In
the past two decades, Zejian Jiang, D. A. Herrero, Chekao Fang,
Chunlan Jiang, Zongyao Wang, Peiyuan Wu, Youqing Ji, Junsheng Fang
and many other mathematicians did a lot of work around this question
as in \cite{Jiang_1, Jiang_2, Jiang_3, Jiang_4, Jiang_5, Jiang_6,
Watatani_1, Watatani_2, Watatani_3}. However, for many operators on
an infinite dimensional Hilbert space, such as multiplication
operators, their commutants may contain no minimal idempotents. For
these operators, direct sums need to be replaced by direct integrals
with some suitable measures. Therefore we briefly recall the related
concepts about ``\textit{direct integrals of Hilbert spaces and
operators}'' after the following result proved by E. A. Azoff, C. K.
Fong and F. Gilfeather in 1976 \cite{Azoff}.

\begin{thm}[Theorem AFG] Let $\mathscr {U}$ be a subalgebra of
$\mathscr {L}(\mathscr {H})$ closed in the strong operator topology.
Then the decomposition $\mathscr {U}\thicksim
\int^{\oplus}_{\Lambda} \mathscr {U}(\lambda)\mu(d\lambda)$ is
maximal if and only if almost all the algebras $\{\mathscr
{U}(\lambda)\}$ are irreducible.
\end{thm}

For the related concepts mentioned here, we follow \cite{Azoff,
Radjavi, Schwartz}. Once and for all, fix a sequence of Hilbert
spaces $\mathscr {H}_1\subset\mathscr
{H}_2\subset\cdots\subset\mathscr {H}_{\infty}$ with $\mathscr
{H}_n$ having dimension $n$ and $\mathscr {H}_{\infty}$ spanned by
the union of Hilbert spaces $\mathscr {H}_n$, for $1\leq n<\infty$.
Next, suppose we have a {\textit{partitioned measure space}}
$(\Lambda,\mu,\{\Lambda_{n}\}^{n=\infty}_{n=1})$. This means that
$\Lambda$ is a separable metric space, $\mu$ is (the completion of)
a regular Borel measure on $\Lambda$, and $\{\Lambda_{\infty},
\Lambda_{1}, \Lambda_{2}, \ldots\}$ is a Borel partition of
$\Lambda$, where some $\Lambda_{n}$s may be of measure zero. We also
assume that $\mu$ is $\sigma$-finite and $\Lambda$ is almost
$\sigma$-compact. Then we form the associated direct integral
Hilbert space $\mathscr {H}=\int^{\oplus}_{\Lambda} \mathscr
{H}(\lambda) \mu(d\lambda)$. This space consists of all (equivalence
classes of) measurable functions $f$ from $\Lambda$ into $\mathscr
{H}_{\infty}$ such that $f(\lambda)\in \mathscr {H}(\lambda)\equiv
\mathscr {H}_n$ for $\lambda\in \Lambda_n$ and
$\int_{\Lambda}\|f(\lambda)\|^2\mu(d\lambda)<\infty$. The element in
$\mathscr {H}$ represented by the function $\lambda\rightarrow
f(\lambda)$ is denoted by $\int^{\oplus}_{\Lambda}f(\lambda)\mu
(d\lambda)$. An  operator $A$ in $\mathscr {L}(\mathscr {H})$ is
said  to be {\textit{decomposable}} if there exists a strongly
$\mu$-measurable operator-valued function $A(\cdot)$ defined on
$\Lambda$ such that $A(\lambda)$ is an operator in $\mathscr
{L}(\mathscr {H}(\lambda))$ and
$(Af)(\lambda)=A(\lambda)f(\lambda)$, for all $f\in \mathscr {H}$.
We write $A\equiv\int^{\oplus}_{\Lambda}A(\lambda)\mu(d\lambda)$ for
the equivalence class corresponding to $A(\cdot)$.  If $A(\lambda)$
is a scalar multiple of the identity on $\mathscr {H}(\lambda)$ for
almost all $\lambda$, then $A$ is said to be {\textit{diagonal}}.
The collection of all diagonal operators is said to be the
{\textit{diagonal algebra}} of $\Lambda$. It is an abelian von
Neumann algebra.

As an application of Theorem AFG, J. Fang, C. Jiang and P. Wu
\cite{Fang} proved that every operator is a direct integral of
irreducible operators. One may ask whether every bounded linear
operator on a Hilbert space can be written as a direct sum of at
most countably many strongly irreducible operators. In \cite{Jiang
Z}, Zejian Jiang gave a negative answer by proving that every
self-adjoint operator without eigenvalues can not be written as a
direct sum of strongly irreducible operators.

An operator $A$ in $\mathscr {L}(\mathscr {H})$ is a direct integral
of strongly irreducible operators if there exists a partitioned
measure space $(\Lambda,\mu,\{\Lambda_n\}^{n=\infty}_{n=1})$ for
$\mathscr {H}$ such that $A$ is decomposable with respect to
$(\Lambda,\mu,\{\Lambda_n\}^{n=\infty}_{n=1})$ and $A(\lambda)\in
\mathscr {L}(\mathscr {H}(\lambda))$ is strongly irreducible almost
everywhere on $\Lambda$.

The main purpose of this paper is to develop a direct integral
theory for a large class of operators on infinite dimensional
Hilbert spaces, which can be considered as a generalization of the
Jordan canonical form theorem for operators on finite dimensional
Hilbert spaces. More precisely, we prove the following theorem.

\begin{thm}[Main Theorem]
Let $\mathscr {H}$ be a separable Hilbert space and let $A\in
\mathscr {L}(\mathscr {H})$. The operator $A$ is similar to a direct
integral of strongly irreducible operators if and only if its
commutant $\{A\}'$ contains a bounded maximal abelian set of
idempotents $\mathscr {P}$.
\end{thm}

This paper is organized as follows. In Section 2, we construct two
examples to show that not every operator is similar to a direct
integral of strongly irreducible operators. In Section 3, we develop
our theory of strongly irreducible decomposition and prove the Main
Theorem. In Section 4, we provide several applications. As an
important application, a large class of spectral operators can be
written as direct integrals of strongly irreducible operators.


\section{Examples Of Operators Which Are Not Similar To Direct
Integrals Of Strongly Irreducible Operators}

In this section, we construct two examples as mentioned above. It is
worth pointing out that both of them can be written as direct
integrals of irreducible operators.

\begin{ex}
Suppose
$$A=\bigoplus\limits^{\infty}_{i=1}A_{i}.\quad
A_i=\left(\begin{array}{cc}
\frac{1}{i}&1\\
0&-\frac{1}{2i}
\end{array}\right)\in M_2(\mathds{C}).$$
Then $A$ is not similar to a direct integral of strongly irreducible
operators. (Notice that each $A_i$ is irreducible in
$M_2(\mathds{C})$.)\\

We calculate the commutant of $A$. If $T_i\in\{A_i\}'$, then
$$T_i=\left(\begin{array}{cc}
t_1&\frac{2i}{3}(t_1-t_2)\\
0&t_2
\end{array}\right), t_1,t_2\in\mathds{C}.$$

If $T\in\{A\}'$, then $T$ can be expressed in the form
$$T=\left(\begin{array}{ccc}
T_{11}&T_{12}&\cdots\\
T_{21}&T_{22}&\cdots\\
\vdots&\vdots&\ddots\\
\end{array}\right),$$
where $T_{ij}A_{j}=A_{i}T_{ij}$.

If $i\neq j$, then by calculation we have $T_{ij}=0$ and
$$T=\left(\begin{array}{ccc}
T_{11}&0&\cdots\\
0&T_{22}&\cdots\\
\vdots&\vdots&\ddots\\
\end{array}\right).$$

If $P$ is an idempotent in $\{A\}'$, then $P$ has the form
$$P=\left(\begin{array}{ccc}
P_{1}&0&\cdots\\
0&P_{2}&\cdots\\
\vdots&\vdots&\ddots\\
\end{array}\right),$$
where $P_{i}$ is an idempotent in $\{A_{i}\}'$.

We obtain that $P_{i}$ has the following four forms
$$\left(\begin{array}{cc}
1&\frac{2i}{3}\\
0&0\\
\end{array}\right),
\left(\begin{array}{cc}
0&-\frac{2i}{3}\\
0&1\\
\end{array}\right),0_{2\times 2},I_{2\times 2}.$$

Since $P\in\{A\}'$ is bounded, $P_i$ must be $0$ or $I$ for all but
finitely many $i\in\mathds{N}$. Hence, all the idempotents in
$\{A\}'$ form the only maximal abelian set of idempotents $\mathscr
{P}$. But this set is unbounded.

If there exists an operator $B$ similar to $A$ and $B$ can be
written as a direct integral of strongly irreducible operators, then
every idempotent in the commutant of $B$ must be a projection. This
is a contradiction because $\mathscr {P}$ is unbounded.
\end{ex}

\begin{ex}
If $\mu$ is the Lebesgue measure supported on interval $[0,1]$.
Define $N_{\mu}$ on $L^2(\mu)$ by
$$(N_{\mu}f)(t)=(z\cdot f)(t)=t\cdot f(t),\ \forall f\in
L^2(\mu).$$ Note that there is no minimal idempotent in
$\{N_{\mu}\}'$.

Let $A$ be represented as a $2\times2$ matrix with operator entries
in the form
$$A=\left(\begin{array}{cc}
N_{\mu}&I\\
0&-\frac{1}{2}N_{\mu}\\
\end{array}\right)
\begin{array}{c}
L^2(\mu)\\
L^2(\mu)\\
\end{array}.$$ We show that $A$
is not similar to a direct integral of strongly irreducible
operators.

If $Q$ is an operator in $\{A\}'$, then $Q$ can be expressed in the
form
$$Q=\left(\begin{array}{cc}
Q_{11}&Q_{12}\\
Q_{21}&Q_{22}\\
\end{array}\right)
\begin{array}{c}
L^2(\mu)\\
L^2(\mu)\\
\end{array}.$$

The equation $QA=AQ$ becomes
$$\left(\begin{array}{cc}
Q_{11}N_{\mu}&Q_{11}-\frac{1}{2}Q_{12}N_{\mu}\\
Q_{21}N_{\mu}&Q_{21}-\frac{1}{2}Q_{22}N_{\mu}\end{array}\right)=
\left(\begin{array}{cc}
N_{\mu}Q_{11}+Q_{21}&N_{\mu}Q_{12}+Q_{22}\\
-\frac{1}{2}N_{\mu}Q_{21}&-\frac{1}{2}N_{\mu}Q_{22}\end{array}\right).$$

The equation $Q_{21}N_{\mu}=-\frac{1}{2}N_{\mu}Q_{21}$ implies
$Q_{21}=0$, by applying Proposition 6.10 in (\cite{Conway}, IX).

The equations $$\left\{\begin{array}{rcl}
Q_{11}N_{\mu}&=&N_{\mu}Q_{11}+Q_{21},\\
Q_{21}-\frac{1}{2}Q_{22}N_{\mu}&=&-\frac{1}{2}N_{\mu}Q_{22},\\
Q_{11}-\frac{1}{2}Q_{12}N_{\mu}&=&N_{\mu}Q_{12}+Q_{22},
\end{array}\right.$$
imply that $Q_{11},Q_{22}\in \{N_{\mu}\}'$ and
$Q_{11}-Q_{22}=N_{\mu}Q_{12}+\frac{1}{2}Q_{12}N_{\mu}$.

If $Q_{11}=Q_{22}=0$, then we have $Q_{12}=0$. Hence for any $Q$ in
$\{A\}'$, $Q_{12}$ is uniquely determined by $Q_{11}$ and $Q_{22}$.
Thus the commutant of $A$ is abelian, but the set of all the
idempotents in $\{A\}'$ is unbounded.

If there exists an operator $B$ similar to $A$ and $B$ can be
written as a direct integral of strongly irreducible operators, then
there is only one maximal abelian von Neumann algebra $\mathscr {D}$
in the commutant of $B$. We obtain the same contradiction as the
above example.

In particular, the collection $\mathscr {P}$ of projections in the
commutant of $A$ is abelian. The operator $N^{(2)}_{\mu}$ and
$\mathscr {P}$ generate the same maximal abelian von Neumann algebra
in $\{A\}'$. We can fix the partitioned measure space by
$N^{(2)}_{\mu}$. Thus as an application of Theorem AFG, $A$ can be
written as a direct integral of irreducible matrices in
$M_2(\mathds{C})$,
$$A=\int^{\oplus}_{[0,1]}\left(\begin{array}{cc}\lambda&1\\
0&-\frac{1}{2}\lambda\end{array}\right)\mu(d\lambda).$$

\end{ex}


\section{Proof Of The Main Theorem}

Notice that for each example in \S 2, the commutant $\{A\}'$ does
not contain a bounded maximal abelian set of idempotents. Therefore
in this section we assume that $\mathscr {H}$ is an infinite
dimensional Hilbert space, $A\in\mathscr {L}(\mathscr {H})$ and
there exists a bounded maximal abelian set of idempotents $\mathscr
{P}$ in $\{A\}'$. We show that $A$ is similar to a direct integral
of strongly irreducible operators.

First we show that there exists an invertible operator $X\in\mathscr
{L}(\mathscr {H})$ such that $XPX^{-1}$ is a projection for each
$P\in\mathscr {P}$. Secondly, we recall how to write a
non-self-adjoint algebra $\mathscr {A}$ in the form of a direct
integral with respect to an abelian von Neumann algebra in the
commutant of $\mathscr {A}$. Finally, we prove the Main Theorem.

The following three lemmas lead us to the first goal as stated
above.

\begin{lem}
The set $\mathscr {P}$ is an abelian Boolean algebra of idempotents.
\end{lem}

\begin{proof}
We verify the definition of Boolean algebra.

(1) We verify that $\mathscr {P}$ is a lattice. For all
$P_1,P_2\in\mathscr {P}$, define the following operations
$$P_1\vee P_2\equiv P_1+P_2-P_1P_2,\quad P_1\wedge P_2\equiv P_1P_2.$$
The operations ``$\vee$'' and ``$\wedge$'' are closed in $\mathscr
{P}$. Thus $\mathscr {P}$ is a lattice.

(2) The set $\mathscr {P}$ includes $0$ and $I$.

(3) Define the complement operation as $P^{c}\equiv I-P,\ \forall
P\in \mathscr {P}$. Hence this operation is still closed in
$\mathscr {P}$.

(4) We verify that $\mathscr {P}$ satisfies the Distributive Law.
For all $P_1,P_2,P_3\in\mathscr {P}$,
$$\begin{array}{rcl}
P_1\wedge(P_2\vee P_3)&=&P_1(P_2+P_3-P_2P_3)\\
&=&P_1P_2+P_1P_3-P_1P_2P_3\\
(P_1\wedge P_2)\vee(P_1\wedge P_3)&=&P_1P_2+P_1P_3-P_1P_2P_3.
\end{array}$$
Therefore
$$P_1\wedge(P_2\vee P_3)=(P_1\wedge P_2)\vee(P_1\wedge
P_3).$$ On the other hand,
$$\begin{array}{rcl}
P_1\vee(P_2\wedge P_3)&=&P_1+P_2P_3-P_1P_2P_3;\\
(P_1\vee P_2)\wedge(P_1\vee P_3)&=&(P_1+P_2-P_1P_2)(P_1+P_3-P_1P_3)\\
&=& P_1+P_2P_3-P_1P_2P_3.\\
\end{array}$$
Therefore
$$P_1\vee(P_2\wedge P_3)=(P_1\vee P_2)\wedge(P_1\vee P_3).$$
\end{proof}

The following lemma shows that $\mathscr {P}$ possesses a nice
property.

\begin{lem} [\cite{Wermer}]
Let $E(\sigma)$ and $F(\eta)$ be two commuting spectral measures on
a Hilbert space $\mathscr {H}$, that is,
$$E(\sigma)F(\eta)=F(\eta)E(\sigma),$$ for every $\sigma$ and
$\eta$. Then there exists an invertible operator $X$ such that
$XE(\sigma)X^{-1}$ and $XF(\eta)X^{-1}$ are projections for every
$\sigma$ and $\eta$.
\end{lem}

In this lemma, a ``spectral measure'' on a Hilbert space $\mathscr
{H}$ means a function $E(\cdot):\Omega\rightarrow\mathscr
{L}(\mathscr {H})$, where $\Omega$ is the $\sigma$-algebra of Borel
sets in the complex plane. Note that for every $\sigma\in\Omega$,
$E(\sigma)$ is an idempotent and the image of $E$ is bounded in the
operator norm. One may see \cite{Wermer} for details. Note that the
image of $E$ is a bounded abelian Boolean algebra of idempotents.
Except for the spectral measure in Lemma 3.2, all spectral measures
mentioned in this article are projection-valued.

Now we want to verify that any similar transformation preserves the
maximality of $\mathscr {P}$ as an abelian set of idempotents.

Let $X\in \mathscr {L}(\mathscr {H})$ be an invertible operator.
Write
$$X\mathscr {P}X^{-1}\equiv\{XPX^{-1}:P\in\mathscr {P}\}.$$

\begin{lem}
The set $X\mathscr {P}X^{-1}$ is a maximal abelian set of
idempotents in $\{XAX^{-1}\}'$.
\end{lem}

\begin{proof}
For any $XPX^{-1}$ in $X\mathscr {P}X^{-1}$, $(XPX^{-1})^2=XPX^{-1}$
implies that $X\mathscr {P}X^{-1}$ is a set of idempotents and it is
easy to verify that $X\mathscr {P}X^{-1}$ belongs to
$\{XAX^{-1}\}'$, since
$XPX^{-1}XAX^{-1}=XPAX^{-1}=XAPX^{-1}=XAX^{-1}XPX^{-1}$. If $Q$ is
an idempotent in $\{XAX^{-1}\}'$ and commutes with every element in
$X\mathscr {P}X^{-1}$, then by definition we have that $X^{-1}QX$ is
in $\{A\}'$ and commutes with each element in $\mathscr {P}$. This
shows that $X^{-1}QX$ is in $\mathscr {P}$. equivalently, $Q\in
X\mathscr {P}X^{-1}$. Therefore $X\mathscr {P}X^{-1}$ is a maximal
abelian set of idempotents in $\{XAX^{-1}\}'$.

\end{proof}

Notice that Lemma 3.2 also applies to the case of bounded abelian
Boolean algebra of idempotents. By applying the above three lemmas,
there exists an invertible operator $X$ such that each element in
$X\mathscr {P}X^{-1}$ is a projection, meanwhile $X\mathscr
{P}X^{-1}$ is still a maximal abelian set of idempotents in
$\{XAX^{-1}\}'$.

Therefore without loss of generality, in the rest of this section,
we assume that each element in $\mathscr {P}$ is a projection while
$\mathscr {P}$ is still a bounded maximal abelian set of idempotents
in $\{A\}'$. Let $W^*(\mathscr {P})$ denote the von Neumann algebra
generated by $\mathscr {P}$. Then the following lemma simplifies the
statements in the second part of this section.

\begin{lem}
The von Neumann algebra $W^*(\mathscr {P})$ is a maximal abelian von
Neumann algebra (m.a.s.a.) in $\{A\}'$.
\end{lem}

The operator $A$ can be expressed in the form of a direct integral
by Theorem AFG. But a ``direct integral'' is a ``relative'' concept,
namely, each (direct integral) decomposition of $A$ depends on an
abelian von Neumann algebra in $\{A\}'$. In this part, we specify
how a (direct integral) decomposition of $A$ depends on the choice
of an abelian von Neumann algebra in $\{A\}'$ and how an abelian von
Neumann algebra in $\{A\}'$ determines a partitioned measure space.
To do this, we recall two results from the classical theory of von
Neumann algebras.

Here are some necessary notations. Let $T$ be an operator in
$\mathscr {L}(\mathscr {H})$. The direct sum of $T$ with itself $n$
times is denoted by $T^{(n)}$, which acts on $\mathscr
{H}^{(n)}\equiv\mathscr {H}\oplus\mathscr
{H}\oplus\cdots\oplus\mathscr {H}$ ($n$ copies). Let $\mathscr {S}$
be a subset of $\mathscr {L}(\mathscr {H})$. Then we write $\mathscr
{S}^{(n)}$ for $\{T^{(n)}\in\mathscr {L}(\mathscr
{H}^{(n)}):T\in\mathscr {S}\}$, $\{\mathscr {S}\}'$ for the
commutant of $\mathscr {S}$ and $[\mathscr {S}]$ for the strongly
closed algebra generated by $\mathscr {S}$.

\begin{lem} [\cite{Rosenthal}]
If $\mathscr {U}$ is an abelian von Neumann algebra, then there
exists a self-adjoint operator $N$ such that $\{N\}''=\mathscr {U}$.
\end{lem}

\begin{lem} [\cite{Conway}]
If $N$ is a normal operator, then there are mutually singular
measures $\mu_{\infty},\mu_1,\mu_2,\ldots$ (some of which may be
zero) such that $$N\cong N^{(\infty)}_{\mu_{\infty}}\oplus
N_{\mu_1}\oplus N^{(2)}_{\mu_{2}}\oplus\cdots.$$ If $M$ is another
normal operator with corresponding measures
$\nu_{\infty},\nu_{1},\nu_{2},\ldots,$ then $N\cong M$ if and only
if $[\mu_{n}]=[\nu_{n}]$ for $1\leq n\leq\infty$.
\end{lem}

For more details, see (\cite{Conway}, Chapter IX, \S 10).

By Lemma 3.5 and Lemma 3.6, there exists a self-adjoint operator $D$
in $W^*(\mathscr {P})$ such that
\begin{enumerate}
\item If $\mu$ is a scalar-valued spectral measure for
$D$, then there are pairwise disjoint Borel sets
$\Delta_{\infty,D},\Delta_{1,D},\Delta_{2,D},\ldots$ such that
$\mu_{n}$ and $\mu|_{\Delta_{n,D}}$ are mutually absolutely
continuous, and
$$\sigma(D)=\Delta_{\infty,D}\cup(\bigcup\limits^{\infty}_{n=1}\Delta_{n,D}).$$
\item The equation $$W^*(\mathscr {P})=\{D\}''\cong
(L^{\infty}(\Delta_{\infty,D},\mu_{\infty}))^{(\infty)}\oplus
(\bigoplus^{\infty}_{n=1}(L^{\infty}(\Delta_{n,D},\mu_{n}))^{(n)})
\eqno{(\dag)}$$ holds, where
$(L^{\infty}(\Delta_{n,D},\mu_{n}))^{(n)}$ acts on
$L^{2}(\mu_{n};\mathscr {H}_{n})$, and
$$L^{2}(\mu_{n};\mathscr {H}_{n})=\int^{\oplus}_{\Delta_{n,D}}
\mathscr {H}(\lambda) \mu(d\lambda),$$  the Hilbert space $\mathscr
{H}(\lambda)=\mathscr {H}_{n}$ is of dimension $n$ for
$\lambda\in\Delta_{n,D}$.
\end{enumerate}

Let $\mathscr {A}(A)$ denote the strongly closed algebra generated
by the operator $A$ and the identity $I$. Choose the triple set
$(\sigma(D),\mu,\{\Delta_{n,D}\}^{n=\infty}_{n=1})$ as the
partitioned measure space. Let $W^*(\mathscr {P})$ be the
corresponding diagonal algebra. Then each operator in $\mathscr
{A}(A)$ is decomposable with respect to
$(\sigma(D),\mu,\{\Delta_{n,D}\}^{n=\infty}_{n=1})$.

Now we introduce the (direct integral) decomposition for $\mathscr
{A}(A)$ with respect to
$(\sigma(D),\mu,\{\Delta_{n,D}\}^{n=\infty}_{n=1})$. As in
\cite{Azoff}, we choose a countable generating set
$\{A_{i}\}^{\infty}_{i=1}$ for $\mathscr {A}(A)$ and fix Borel
representatives
$$\lambda\rightarrow A_{i}(\lambda)$$ for their corresponding
decompositions.

For each $\lambda\in\sigma(D)$, denote by $\mathscr {A}(A)(\lambda)$
the strongly closed algebra generated by the
$\{A_{i}(\lambda)\}^{\infty}_{i=1}$. We write $$\mathscr {A}(A)\sim
\int^{\oplus}_{\sigma(D)}\mathscr {A}(A)(\lambda)\mu(d\lambda)$$

\vspace{-1 cm}\hfill(\ddag)\\
\\
and call this the decomposition of $\mathscr {A}(A)$ with respect to
$(\sigma(D),\mu,\{\Delta_{n,D}\}^{n=\infty}_{n=1})$.

In \cite{Azoff}, the authors proved that the algebras $\{\mathscr
{A}(A)(\lambda)\}_{\lambda\in\sigma(D)}$ in the above statements are
independent of the generating set $\{A_{i}\}^{\infty}_{i=1}$, up to
a set of measure zero in $\sigma(D)$. The algebra $\mathscr {A}(A)$
completely determines the algebras $\mathscr {A}(A)(\lambda)$.
However, the converse is not true in general.

The following two basic results are well known: \begin{enumerate}
\item An operator acting on a direct integral of Hilbert spaces is
decomposable if and only if it commutes with the corresponding
diagonal algebra (\cite{Schwartz}).
\item Every abelian von Neumann algebra is (unitarily equivalent to)
an essentially unique diagonal algebra (\cite{Schwartz}).
\end{enumerate}

\begin{defn}
A strongly closed subalgebra $\mathscr {A}$ of $\mathscr
{L}(\mathscr {H})$ is said to be strongly irreducible (resp.\
irreducible) if any idempotent (resp.\ projection) $P\in \mathscr
{L}(\mathscr {H})$ satisfying $PT=TP$ for all $T\in\mathscr {A}$ is
trivial.
\end{defn}

\begin{lem}
An operator $T\in\mathscr {L}(\mathscr {H})$ is strongly irreducible
if and only if $\mathscr {A}(T)$ is strongly irreducible.
\end{lem}

\begin{proof}
If $T$ is strongly irreducible and $P\in \mathscr {L}(\mathscr {H})$
is an idempotent satisfying $PS=SP$ for all $S\in\mathscr {A}(T)$,
then $PT=TP$ implies $P=0$ or $P=I$.

If $T$ is strongly reducible, namely, there exists a nontrivial
idempotent $P$ such that $PT=TP$, then for every polynomial in $T$
denoted by ${\rm p}(T)$, we have $$P\ {\rm p}(T)={\rm p}(T)P.$$ Thus
$P\mathscr {A}(T)=\mathscr {A}(T)P$, since the set $\{{\rm p}(T):
{\rm p}=a~polynomial\}$ forms a dense subset of $\mathscr {A}(T)$ in
the strong operator topology. This is a contradiction.
\end{proof}

\begin{lem}[\cite{Schwartz}]
If $B\in\mathscr {L}(\mathscr {H})$ is decomposable and
$B\equiv\int^{\oplus}_{\Lambda}B(\lambda)\mu(d\lambda)$, then
$$\|B\|={\mu-\underset{\lambda\in\Lambda}{\rm ess.}\sup}\|B(\lambda)\|.$$
\end{lem}

The following ``Principle of Measurable Choice'' will be used
repeatedly in our discussion. Von Neumann's original proof can be
found in \S 16 of \cite{von Neumann}. In this article, we follow
\cite{Azoff}.

\begin{lem}  Suppose $\Lambda$ is a partitioned measure space, $\mathscr{Y}$
is a complete metric space, and $\mathscr{E}$ is a Borel relation
contained in $\Lambda\times \mathscr{Y}$. Then the domain of
$\mathscr{E}$ is a measurable subset of $\Lambda$. Furthermore,
there exists a Borel function whose domain almost coincides with the
domain of $\mathscr{E}$ and whose graph is contained in
$\mathscr{E}$.
\end{lem}

By what we have discussed in the second part of this section, we
assume that all the Hilbert spaces $\mathscr {H}(\lambda)$ for
$\lambda\in\sigma(D)$ are of the same dimension $n$ (denoted by
$\mathscr {H}_{n}$) in the rest of this section without loss of
generality. Denote by $\mathscr {Q}(\mathscr {H}_{n})$ the
collection of idempotents in $\mathscr {L}(\mathscr {H}_{n})$, by
$\mathscr {C}(\mathscr {H}_{n})$ the collection of contractions in
$\mathscr {L}(\mathscr {H}_{n})$. Let $\{A_{i}\}^{\infty}_{i=1}$ be
a countable generating set for $\mathscr {A}(A)$ with $\|A_{i}\|\leq
1$. Suppose that $\Lambda_1$ and $\Lambda_2$ are topological spaces.
The map $\pi_{i}:\Lambda_1\times\Lambda_2\rightarrow\Lambda_{i}$ is
defined by $\pi_{i}(\lambda_1,\lambda_2)=\lambda_{i}$ for $i=1,2$.

\begin{lem} The set
$$\mathscr{E}\equiv\{(\lambda,Q)\in\sigma(D)\times \mathscr {Q}(\mathscr {H}_{n}):
A_{i}(\lambda)Q=QA_{i}(\lambda),\ Q\neq 0, I,\ i=1,2,\ldots\}$$ is a
Borel relation, $\pi_{1}(\mathscr{E})$ is a measurable subset of
$\sigma(D)$.
\end{lem}

\begin{proof}
By the (direct integral) decomposition of $\mathscr {A}(A)$
mentioned in the second part of this section, the set
$\{A_{i}(\lambda)\}^{\infty}_{i=1}$ generate $\mathscr
{A}(A)(\lambda)$ for almost every $\lambda\in\sigma(D)$. By Lemma
3.9, $\|A_{i}(\lambda)\|\leq 1$ holds almost everywhere on
$\sigma(D)$. Define
$$\mathscr {Q}_k(\mathscr {H}_{n}) \equiv
\{Q\in\mathscr {L}(\mathscr {H}_{n}):Q^2=Q,\|Q\|\leq k\},\
k=1,2,3,\ldots$$ then we have $$\mathscr {Q}(\mathscr
{H}_{n})=\bigcup\limits^{\infty}_{k=1}\mathscr {Q}_k(\mathscr
{H}_{n}).$$ Define
$$\mathscr{E}_k\equiv\{(\lambda,Q)\in\sigma(D)\times \mathscr {Q}_k(\mathscr {H}_{n}):
A_{i}(\lambda)Q=QA_{i}(\lambda), Q\neq 0, I,\ i=1,2,\ldots\}.$$ Then
$$\mathscr{E}=\bigcup\limits^{\infty}_{k=1}\mathscr{E}_k,\quad \pi_{1}(\mathscr{E})=
\bigcup\limits^{\infty}_{k=1} \pi_{1}(\mathscr{E}_k).$$ Since
composition is strongly continuous on bounded subset of $\mathscr
{L}(\mathscr {H}_{n})$, $\mathscr{E}_k$ is a Borel relation.
Moreover $\pi_{1}(\mathscr{E})= \bigcup\limits^{\infty}_{k=1}
\pi_{1}(\mathscr{E}_k)$ is a measurable subset of $\sigma(D)$.
\end{proof}

\begin{proof}[{Proof of Theorem 1.2 (Main Theorem)}] We prove the ``if''
part. The ``only if'' part is clear.

By (\ddag) (after Lemma 3.6), we write $\mathscr {A}(A)$ in the form
of a direct integral, hence we have
$$A=\int^{\oplus}_{\sigma(D)}A(\lambda)\mu(d\lambda).$$ We want to
prove $A(\lambda)$ is strongly irreducible almost everywhere on
$\sigma(D)$. Equivalently, we need to prove that the collection of
$\lambda$, for which $A(\lambda)$ is strongly reducible, is a set of
measure zero. More precisely, we prove that for the set
$$\mathscr{E}=\{(\lambda,Q)\in\sigma(D)\times \mathscr {Q}(\mathscr {H}_{n}):
A_{i}(\lambda)Q=QA_{i}(\lambda), Q\neq 0, I,\ i=1,2,\ldots\},$$ the
set $\pi_{1}(\mathscr{E})$ is of measure zero. By
$\pi_{1}(\mathscr{E})= \bigcup^{\infty}_{k=1}
\pi_{1}(\mathscr{E}_k)$, we prove that every
$\pi_{1}(\mathscr{E}_k)$ is of measure zero.

Applying Lemma 3.10, for a fixed positive integer $k$, we get a
Borel function $P_k(\cdot)$ defined on almost all of the domain
$\pi_{1}(\mathscr{E}_k)$ such that $P_{k}(\lambda)\in\mathscr
{Q}_{k}(\mathscr {H}_{n})$ is nontrivial for almost every
$\lambda\in\pi_{1}(\mathscr{E}_k)$. Let $P_k(\lambda)$ be zero for
every $\lambda\in\sigma(D)\backslash\pi_{1}(\mathscr{E}_k)$ and take
$P_{k}=\int^{\oplus}_{\sigma(D)}P_{k}(\lambda)\mu(d\lambda)$.

For almost every $\lambda\in\sigma(D)$, we have $P_k(\lambda)\in
\{\mathscr {A}(A)(\lambda)\}'$. By $P^2_k(\lambda)=P_k(\lambda)$ for
almost every $\lambda\in\sigma(D)$, we know that $P_k$ is a
decomposable idempotent and $P_k\in\{\mathscr {A}(A)\}'$.

We show that $P_{k}$ is in $\mathscr {P}$. For almost every
$\lambda\in\sigma(D)$, $P_k(\lambda)$ commutes with the identity $I$
or 0 in $\mathscr {L}(\mathscr {H}_{n})$. This implies that
$P_k\in\{W^*(\mathscr {P})\}'$. Because every idempotent in
$\{A\}'\cap\mathscr {P}'$ belongs to $\mathscr {P}$, we obtain
$P_k\in \mathscr {P}$. This means that $P_k(\lambda)$ is trivial
almost everywhere on $\sigma(D)$. Therefore by the definition of
$P_k(\cdot)$, we have $P_k=0$ and $\pi_{1}(\mathscr{E}_k)$ has
measure zero. Thus $\pi_{1}(\mathscr{E})$ has measure zero.

So $\mathscr {A}(A)(\lambda)$ is strongly irreducible almost
everywhere on $\sigma(D)$. In the second part of this section, we
mentioned that the (direct integral) decomposition of $\mathscr
{A}(A)$ is independent of the generating set up to a set of measure
zero, and the collection of all polynomials in $A(\lambda)$ is a
countable generating set of $\mathscr {A}(A)(\lambda)$ for $\lambda$
almost everywhere in $\sigma(D)$. By Lemma 3.8, the operator
$A(\lambda)$ is strongly irreducible almost everywhere on
$\sigma(D)$.
\end{proof}


\section{Applications Of The Main Theorem}

In this section, we study some classes of operators by Theorem 1.2.
The first class is normal operators.

\begin{prop}
If $N\in\mathscr {L}(\mathscr {H})$ is a normal operator, then each
maximal abelian set of projections in $\{N\}'$ is a bounded maximal
abelian set of idempotents in $\{N\}'$.
\end{prop}

\begin{proof}
Let $\mathscr {M}$ be a maximal abelian von Neumann algebra in
$\{N\}'$. By $N\in \{N\}'\cap\mathscr {M}'$, we have $N\in \mathscr
{M}$, otherwise $\mathscr {M}$ is not maximal. If $N_{1}$ is a
normal operator in $\mathscr {M}'$, then $N_{1}N=NN_{1}$ and
$N_{1}\in\{N\}'$. Hence $\mathscr {M}$ is a maximal abelian von
Neumann algebra in $\mathscr {L}(\mathscr {H})$. Let $\mathscr {Q}$
be the lattice of all the projections in $\mathscr {M}$. If
$P\in\{N\}'$ is an idempotent commuting with every element in
$\mathscr {Q}$, then we have $PM_1=M_1P$ for all $M_1\in\mathscr
{M}$. By (\cite{Conway}, IX, 6.10), there exists a projection $P'$
with ${\rm ran} P={\rm ran} P'$ and $P'M_1=M_1P'$ for all
$M_1\in\mathscr {M}$. Since $\mathscr {M}$ is maximal in $\mathscr
{L}(\mathscr {H})$, we have $P'\in \mathscr {M}$. Thus $PP'=P'P$.
This implies $P=P'$. So $\mathscr {Q}$ is a bounded maximal abelian
set of idempotents in $\{N\}'$.
\end{proof}

\begin{cor}
Every normal operator is a direct integral of strongly irreducible
operators.
\end{cor}

In 1954, N. Dunford \cite{Dunford} introduced the concept of
\textit{spectral operator} on complex Banach space. A bounded linear
operator $T$ is spectral if and only if it has a canonical
decomposition of the form $$T=S+R,$$ where $S$ is a scalar type
operator and $R$ is a quasinilpotent operator commuting with $S$.
($\sigma(R)=\{0\}$.) A scalar type operator is a spectral operator
with resolution of the identity $I$ which satisfies the equation
$$S=\int_{\sigma(S)}\lambda E(d\lambda).$$

The scalar part $S$ of $T$ and the racial part $R$ of $T$ are
uniquely determined by $T$.

On a Hilbert space $\mathscr {H}$, if $T\in\mathscr {L}(\mathscr
{H})$ is spectral, then there exists an invertible operator
$X\in\mathscr {L}(\mathscr {H})$ such that $XTX^{-1}$ has the
canonical decomposition of the form
$$XTX^{-1}=S+R,$$ where $S$ is a normal operator and $R$ is a
quasinilpotent operator commuting with $S$.

As an application of Theorem 1.2 and by the method introduced in the
proof of the following proposition, many spectral operators in
$\mathscr {L}(\mathscr {H})$ are similar to direct integrals of
strongly irreducible operators.

We denote by $\mu$ a regular Borel measure with compact support.
Define $N_{\mu}$ on $L^2(\mu)$ by
$$(N_{\mu}f)(t)=t\cdot f(t),\ \forall f\in L^2(\mu),\
t\in\textrm{spt}(\mu).$$

On the Hilbert space $(L^2(\mu))^{(2)}$, Define an operator $R$ in
the form
$$R=\left(\begin{array}{cc}
0&M_{\phi}\\
0&0\\
\end{array}\right)
\begin{array}{c}
L^2(\mu)\\
L^2(\mu)\\
\end{array},\quad
\phi\in L^{\infty}(\mu),$$ where $M_{\phi}f=\phi\cdot f,\ \forall
f\in L^2(\mu)$, then we have the following result.

\begin{prop}
Let $T=N^{(2)}_{\mu}+R$. Then $T$ is similar to a direct integral of
strongly irreducible operators.
\end{prop}

First we prepare two lemmas.

\begin{lem}
If $M_{\phi}=I$, then $T$ is similar to a direct integral of
strongly irreducible operators.
\end{lem}

\begin{proof}
By Theorem 1.2, we only need to prove that there exists a bounded
maximal abelian set of idempotents in $\{T\}'$. Let $E(\cdot)$ be a
spectral measure of $N_{\mu}$ and $\mathscr {P}$ be the set of all
the spectral projections of $N_{\mu}$. We want to prove that
$\mathscr {P}^{(2)}$ is a bounded maximal abelian set of idempotents
in the commutant of $T$.

Every spectral projection of $N^{(2)}_{\mu}$ has the form
$$E^{(2)}(\sigma)= \left(\begin{array}{cc} E(\sigma)&0\\
0&E(\sigma) \end{array}\right),$$ where $\sigma$ is a Borel subset
of $\sigma(N_{\mu})$.

Let $Q$ be an idempotent in $\{T\}'$ which commutes with every
projection in $\mathscr {P}^{(2)}$. Then $Q$ commutes with
$N^{(2)}_{\mu}$. We obtain
$$QR=Q(T-N^{(2)}_{\mu})=(T-N^{(2)}_{\mu})Q=RQ.$$

By $QN^{(2)}_{\mu}=N^{(2)}_{\mu}Q$, we have that $Q$ can be
expressed as a $2\times 2$ matrix,
$$Q=\left(\begin{array}{cc}M_{\phi_{11}}&M_{\phi_{12}}\\
M_{\phi_{21}}&M_{\phi_{22}}\end{array}\right),\quad \phi_{ij}\in
L^{\infty}(\mu).$$

The equation $QR=RQ$ becomes
$$\left(\begin{array}{cc}M_{\phi_{11}}&M_{\phi_{12}}\\
M_{\phi_{21}}&M_{\phi_{22}}\end{array}\right)
\left(\begin{array}{cc}0&I\\
0&0\end{array}\right)=
\left(\begin{array}{cc}0&I\\
0&0\end{array}\right)
\left(\begin{array}{cc}M_{\phi_{11}}&M_{\phi_{12}}\\
M_{\phi_{21}}&M_{\phi_{22}}\end{array}\right),$$
$$\left(\begin{array}{cc}0&M_{\phi_{11}}\\
0&M_{\phi_{21}}\end{array}\right) = \left(\begin{array}{cc}M_{\phi_{21}}&M_{\phi_{22}}\\
0&0\end{array}\right).$$

We obtain $M_{\phi_{11}}=M_{\phi_{22}}$ and $M_{\phi_{21}}=0$.

By $Q^{2}=Q$, we have $M^{2}_{\phi_{11}}=M_{\phi_{11}}$ and
$$\left\{\begin{array}{crcl}
&M_{\phi_{11}}M_{\phi_{12}}+M_{\phi_{12}}M_{\phi_{22}}&=&M_{\phi_{12}}\\
\Rightarrow &(2M_{\phi_{11}}-I)M_{\phi_{12}}&=&0\\
\Rightarrow &(2M_{\phi_{11}}-I)^{2}M_{\phi_{12}}&=&0\\
\Rightarrow &M_{\phi_{12}}&=&0
\end{array}\right..$$

Hence $Q$ is a spectral projection of $N^{(2)}_{\mu}$. Thus the set
$\mathscr {P}^{(2)}$ is a bounded maximal abelian set of idempotents
in the commutant of $T$.

By Theorem 1.2, $T$ is similar to a direct integral of strongly
irreducible operators.
\end{proof}

\begin{lem}
If $M_{\phi}$ is a projection, then $T$, defined as above, can be
written as a direct integral of strongly irreducible operators.
\end{lem}

\begin{proof}
If $M_{\phi}$ is a projection, then there exists a Borel subset
$\Delta$ of $\sigma(N_{\mu})$ such that $\phi$ is the characteristic
function on $\Delta$.

Write $\mu_{1}\equiv\mu|_{\Delta}$ and
$\mu_{2}\equiv\mu|_{\sigma(N_{\mu})\backslash\Delta}$. We may
consider $N_{\mu}=N_{\mu_{1}}+N_{\mu_{2}}$, where $N_{\mu_{1}}$
acting on $L^2(\Delta)$, $N_{\mu_{2}}$ on
$L^2(\sigma(N_{\mu})\backslash\Delta)$. Let $\mathscr {P}_{i}$ be
the set of all the spectral projection of $N_{\mu_{i}}$. We prove
that the set $\mathscr {P}^{(2)}_{1}\oplus\mathscr
{P}_{2}\oplus\mathscr {P}_{2}$ is a bounded maximal abelian set of
idempotents in $\{T\}'$.

Write $T=T_{1}\oplus T_{2}$, where
$$T_{1}=\left(\begin{array}{cc}
N_{\mu_{1}}&I\\
0&N_{\mu_{1}}
\end{array}\right)
\begin{array}{c}
L^{2}(\mu_{1})\\
L^{2}(\mu_{1})\\
\end{array},\
T_{2}=\left(\begin{array}{cc}
N_{\mu_{2}}&0\\
0&N_{\mu_{2}} \end{array}\right)
\begin{array}{c}
L^{2}(\mu_{2})\\
L^{2}(\mu_{2})\\
\end{array}.
$$

We can calculate that $$\{T\}'=\{X_1\oplus X_2:X_1\in\{T_1\}',
X_2\in\{T_2\}'\}.$$

Meanwhile, every idempotent $Q$ in $\{T\}'$ is of the form
$Q=Q_1\oplus Q_2$. If $Q$ commutes with every projection in
$\mathscr {P}^{(2)}_{1}\oplus\mathscr {P}_{2}\oplus\mathscr
{P}_{2}$, then we have $Q_1\in \mathscr {P}^{(2)}_1$ and $Q_2\in
\mathscr {P}_2\oplus\mathscr {P}_{2}$.  Thus $\mathscr
{P}^{(2)}_{1}\oplus\mathscr {P}_{2}\oplus\mathscr {P}_{2}$ is a
bounded maximal abelian set of idempotents in $\{T\}'$.

Therefore, by Theorem 1.2, $T$ is similar to a direct integral of
strongly irreducible operators.
\end{proof}

\begin{proof}[Proof of Proposition 4.3] For $\phi\in L^{\infty}(\mu)$,
denote by $\Delta$ the Borel subset of $\sigma(N_{\mu})$ such that
$\phi(\lambda)\neq 0$ for almost every $\lambda$ in $\Delta$ and
$\phi(\lambda)=0$ for almost every $\lambda$ in
$\sigma(N_{\mu})\backslash\Delta$. Let $\phi_{1}$ be
$\phi|_{\Delta}$. As the above lemma, we can write $T=T_{1}\oplus
T_{2}$, where
$$T_{1}=\left(\begin{array}{cc}
N_{\mu_{1}}&M_{\phi_{1}}\\
0&N_{\mu_{1}}
\end{array}\right)
\begin{array}{c}
L^{2}(\mu_{1})\\
L^{2}(\mu_{1})\\
\end{array},\
T_{2}=\left(\begin{array}{cc}
N_{\mu_{2}}&0\\
0&N_{\mu_{2}} \end{array}\right)
\begin{array}{c}
L^{2}(\mu_{2})\\
L^{2}(\mu_{2})\\
\end{array}.$$

Then by a similar calculation, we obtain that $\mathscr
{P}^{(2)}_{1}\oplus\mathscr {P}_{2}\oplus\mathscr {P}_{2}$ is a
bounded maximal abelian set of idempotent in the commutant of $T$.
\end{proof}

If we assume $$R=\left(\begin{array}{cc}0&M_{\phi_{12}}\\
M_{\phi_{21}}&0\end{array}\right),\ \phi_{12}\phi_{21}=0,\
\phi_{ij}\in L^{\infty}(\mu),$$ then $R^2=0$. By a similar method,
we can obtain that $T=N^{(2)}_{\mu}+R$ is similar to a direct
integral of strongly irreducible operators.

In this way, we can construct a large collection of spectral
operators in
$$M_n(L^{\infty}(\mu))\subset \mathscr
{L}(L^{2}(\mu)^{(n)}),$$ which are similar to direct integrals of
strongly irreducible operators. But there exist operators in
$M_n(L^{\infty}(\mu))$ which are not similar to direct integrals of
strongly irreducible operators. In \S 2, we formulate such examples.

As applications, by Chapter 3 of \cite{Jiang_5}, we have the
following corollaries.

\begin{cor}
Every Cowen-Douglas operator is similar to a direct integral of
strongly irreducible operators.
\end{cor}

\begin{cor}
Operators which are similar to direct integrals of strongly
irreducible operators form a dense subset of $\mathscr {L}(\mathscr
{H})$ in the operator norm.
\end{cor}


\subsection*{Acknowledgment}
The authors are grateful to Professor Edward Azoff for his kind
explanation of certain details in \cite{Azoff} and for his helpful
discussion. The corresponding author also wants to express his
gratitude to Professor George Elliott and Professor Junsheng Fang
for their suggestions on writing this paper.

\end{document}